\documentclass[final]{elsarticle}

\usepackage[hidelinks]{hyperref}
\usepackage{algorithm}
\usepackage[noend]{algpseudocode}
\usepackage{amsmath, amssymb, amsthm, latexsym, booktabs, lscape, multicol}
\usepackage[capitalize]{cleveref}

\biboptions{sort&compress}

\newtheorem{prev}{Theorem}
\newtheorem{theorem}{Theorem}
\newtheorem{lemma}{Lemma}
\newtheorem{prop}{Proposition}

\newtheorem{guess}{Conjecture}

\crefname{prev}{Theorem}{Theorems}

\def\leq{\leqslant}
\def\geq{\geqslant}

\def\emptyset{\mbox{{\rm \O}}}
\def\bar{\overline}
\def\gcd{{\rm gcd}}
\def\sup{{\rm sup}}
\def\Z{\mathbb{Z}}

\def\R{\mathbb{R}}

\makeatletter
\def\ps@pprintTitle{%
 \let\@oddhead\@empty
 \let\@evenhead\@empty
 \def\@oddfoot{\centerline{\thepage}}%
 \let\@evenfoot\@oddfoot}
\makeatother

\begin{document}

\begin{frontmatter}

\title{Sequences of Integers with Three Missing Separations\tnoteref{funding}}
\tnotetext[funding]{This work was supported by the National Science Foundation [grant number DMS-1600778]; and NASA [grant number NASA MIRO NX15AQ06A].}

\author[csula]{Daphne~Der-Fen~Liu \corref{cor1}}
\ead{dliu@calstatela.edu}

\author[csula]{Grant Robinson \corref{cor2}}
\ead{grobins4@calstatela.edu}

\address[csula]{Department of Mathematics, California State University Los Angeles, \\5151 State University Drive, Los Angeles, CA 90032-8204}
\cortext[cor1]{Principal corresponding author}
\cortext[cor2]{Corresponding author}

\begin{abstract} 
Let $D$ be a set of positive integers. We study the maximum density $\mu(D)$ of integral sequences 
in which the  separation between any two terms does not fall in $D$. The sets $D$ considered in this article are 
mainly of the form $\{1,j,k\}$. 
The closely related function $\kappa(D)$, the parameter involved in the ``lonely runner conjecture,'' is also investigated. Exact values of $\kappa(D)$ and $\mu(D)$ are found for many families of $D=\{1,j,k\}$. In particular, we prove that the boundary conditions in some earlier results of Haralambis (1977) are sharp. Consequently, our results declaim two recent conjectures of Carraher et al. (2016).  
\end{abstract}

\begin{keyword}
    Density of sequences \sep lonely runner conjecture \sep distance graphs \sep fractional coloring \MSC 05C15 \sep 11B05
\end{keyword}

\end{frontmatter}

\section{Introduction}

Let $D$ be a set of positive integers called a $D$-set. 
A sequence $S$ of non-negative integers is called a {\em $D$-sequence} if $|x-y| \not\in D$ for any $x, y \in S$.  Denote $|S \cap \{0,1,2,\cdots,n\}|$ as $S[n]$. The upper density $\bar{\delta}(S)$ and the lower density $\underline{\delta}(S)$ of $S$ are defined, respectively, by
$\bar{\delta}(S)=\bar{\rm lim}_{n \to \infty} S[n]/(n+1)$ and
$\underline{\delta}(S)=\underline{\rm lim}_{n \to \infty} S[n]/(n+1).$
We say $S$ has a density if
$\bar{\delta}(S)=\underline{\delta}(S)$, and we denote this common value as $\delta(S)$.  The parameter of interest is the {\it density of $D$}, denoted by $\mu(D)$, and defined as 
\[
\mu(D):= {\rm sup} \ \{\ \delta(S): \ S \mbox{\rm \ is a $D$-sequence}\}.
\]

The parameter $\mu(D)$ is closely related to another parameter of $D$ involved in the so called ``lonely runner conjecture". For a real number $x$, let $||x||$ denote the minimum distance from $x$ to an integer, that is $||x|| = \min \{\lceil x \rceil - x, x - \lfloor x \rfloor\}$.  For any real $t$, denote by $||tD||$ the smallest value $||td||$ among all $d \in D$.  The {\it kappa value of $D$}, denoted by $\kappa(D)$, is the supremum of $||tD||$ among all real $t$.  That is,
\[
\kappa(D) := \ \sup \{|| tD || \colon t \in \R \}.
\]
Wills \cite{W} conjectured that $\kappa(D) \geq 1/(|D|+1)$ is true for all finite sets $D$.  This conjecture was named the {\it lonely runner conjecture} by Bienia et al. \cite{lonely}.  Suppose $m$ runners run laps on a circular track of unit circumference. Each runner maintains a constant speed, and the speeds of all the runners are distinct. A runner is called {\it lonely} if the distance on the circular track between him and every other runner is at least $1/m$. Equivalently, Wills' conjecture asserts that for each runner, there exists some time $t$ when he becomes lonely.  The conjecture has been proved true for $|D| \leq 6$ (cf. \cite{7runners, 6runners1, C3, cp}), and remains open for $|D| \geq 7$.

The problem of determining $\mu(D)$ was initially posed by Motzkin in an unpublished problem collection (cf. \cite{CG}). In 1975, Cantor and Gordon \cite{CG} proved that  
\begin{equation}
\mu(D) \geq \kappa(D),  \label{eq:1}
\end{equation}
the fundamental connection between these two concepts, by showing that a ``good time'', that is a $t \in \R$ such that $||tD||$ is maximized, can be used to create a $D$-sequence (cf.  
\cref{alg:kappa}). It was also proved that, for any finite $D$-set, 
$\mu(D)$ exists and this maximal density is rational and achieved by a periodic sequences with period length at most $2^{\max(D)}$. Using extremal walks in finite digraphs, Carraher et al. \cite{conjecture} proved a similar result with periodic sequence of length at most $(\max(D)2^{\max(D)})$. And finally,  for $D = \{a, b\}$,   
Cantor and Gordon \cite{CG} showed that $\kappa(D) = \mu(D) = \lfloor \frac{a+b}{2} \rfloor/(a+b)$.

For $D$-sets with more than two elements, $\mu (D)$ and $\kappa(D)$ have  been investigated intensively (cf. \cite{conjecture, chz,clz, chen2, chen3, chen4, cch, Collister, Gupta, H, surveyL,LZ,PT,WL,pattern}). Readers are referred to the survey \cite{surveyL}.  For three-element $D$-sets, if $D = \{a, b, a+b\}$,  Liu and Zhu \cite{LZ} showed that $\kappa(D) = \mu(D)$,  and determined the exact value. 

\begin{prev}{\rm \cite{LZ}}
\label{M}
Suppose $D = \{a,b,a+b\}$ with {\rm gcd}$(a,b)=1$. Then
\[
\mu(D) = \kappa(D) = \max\left\{\frac{ \ \lfloor\frac{2b+a}{3}\rfloor \ }{2b+a},\ \frac{\ \lfloor\frac{2a+b}{3}\rfloor \ }{2a+b}\right\}.
\]
\end{prev}
\noindent
This confirmed a conjecture of Rabinowiz and Proulz \cite{RP}, who had shown one direction of the equality in \cref{M} in their study of the asymptotic behavior of the channel assignment problem.  
The same inequality had also been discovered  independently by Gupta \cite{Gupta}.  

For the family of sets $D=\{1, j, k\}$, the values of $\mu(D)$ were initially studied by Haralambis \cite{H}.
By considering different parities of $j$, the author established the following two results:

\begin{prev} {\rm \cite{H}}
\label{H-odd}
If $D=\{1, j, k\}$, where $1 < j < k$ and $j$ is odd, then $\mu(D) = 1/2$ if $k$ is odd. Otherwise,
$
\mu(D) = \frac{k}{2(k+j)},
$
provided that $k \geq j(j-1)/2$.
\end{prev}

\begin{prev} {\rm \cite{H}}
\label{H-even}
If $D=\{1, j, k\}$, where $1 < j < k$, $j$ is even, and $k=n(j+1) + \overline{k}$ where $0 \leq \overline{k} \leq j$, then $\mu(D) = \frac{j}{2(j+1)}$ if $\overline{k}=1$ or $j$. Otherwise,
\[
\mu(D) =
\frac{nj/2+ [\overline{k}/2]}{k+1},
\]
provided that $n \geq (j-\overline{k}-2)/2$ when $\overline{k}$ is even, and $n \geq (2j-\overline{k}-3)/2$ when $\overline{k}$ is odd.
\end{prev}

Recently, Carraher et al. \cite{conjecture}, using a local discharging method, determined the values of $\mu(\{1,j,k\})$ for many $j$ and $k$.  Among their results, the following conjectures were posed:

\begin{guess} {\rm (Conjecture 26.  \cite{conjecture})}
\label{conj-odd}
Let $D =\{1, j, k\}$ where $j$ is odd, $j \geq 3$, and $k$ is even.  Then
$
\mu(D) = \frac{k}{2(k+j)},
$
provided that $k \geq 3j$.
\end{guess}

\begin{guess} {\rm (Conjecture 29. \cite{conjecture})}
\label{conj-even}
Let $D = \{1, j, k\}$, where both $j$ and $k$ are even, and $k > j$.  Then
$
\mu(D) = \frac{j}{k+j}.
$
\end{guess}

It was proved in \cite{conjecture} that \cref{conj-odd} is true when $j=3,5,7$. Furthermore, \cref{H-odd} implies that it also holds for $k \geq j(j-1)/2$.  

In Section 3 of this article, we show that \cref{conj-odd} is not always true by establishing the following two theorems: 

\begin{theorem}
\label{4n+1}
Let $D = \{1,j,k\}$ with $j=4n+1$, $k$ even, and $k \leq j(j-1)/2$. 
If $k=m(j+1) + \overline{k}$ with $\overline{k}$ as described below, then 
\[
\kappa(D) = \mu(D) = 
\begin{cases}
\frac{n}{2n+1}      & \textit{if }\overline{k} = 2n, 2n+2\\
\frac{k-2m}{2(k+1)} & \textit{if } \max\{0, 2(n-m)\} \leq \overline{k} \leq 2(n-1).
\end{cases}
\]
\end{theorem}

\begin{theorem}
\label{4n+3}
Let $D = \{1,j,k\}$ with $j=4n+3$, $k$ even, and $k \leq j(j-1)/2$.
If $k=m(j+1) + \overline{k}$ with $\overline{k}$ as described below, then
\[
\kappa(D) = \mu(D) = 
\begin{cases}
\frac{n}{2n+1}      & \textit{if }\overline{k} = 2(n-m) \  \textit{and} \ m \leq 2n\\
\frac{k-2m}{2(k+1)} & \textit{if } \max\{0, 2(n-m+1)\} \leq \overline{k} \leq 2n \\
\frac{2n+1}{4n+4}   & \textit{if }\overline{k} = 2(n+1).
\end{cases}
\]
\end{theorem}

These results determine the values of $\mu(D)$ for some families of the set $D$ in \cref{conj-odd} with $3j \leq k \leq j(j-1)/2$. Our results reveal that the boundary condition $k \geq j(j-1)/2$ required in \cref{H-odd} is sharp (i.e., the best possible).

For \cref{conj-even}, one direction  of the equality  
was established in \cite{conjecture}, while \cref{H-even} implies that the other direction is not always true. 
In Section 4 of this article, we prove the following theorem:

\begin{theorem}
\label{new}
Let $D=\{1, j, k\}$, where $j$ is even, $k=m(j+1)+3$, and $1 \leq m \leq j-3$.  
Then
\[
\mu(D) = \kappa(D) =\frac{j(m+1)}{2(j+k)}. 
\]
\end{theorem}

This result shows that the boundary conditions in \cref{H-even} are also sharp and provides more counterexamples to  \cref{conj-even}. 

For general 3-element $D$-sets, Gupta \cite{Gupta} extended \cref{H-odd} to a similar formula for $D=\{i, j, k\}$ when $j$ is odd. 
The author also proved a lower bound of $\mu(D)$ when $j$ is even for most values of $k$, and showed that the bound is sharp for some cases. When $i=1$, these results recover exactly \cref{H-even,H-odd}, thus leaving the same $D$-sets undetermined. 

After proving \cref{4n+1,4n+3,new} in Sections 3 and 4,  in Section 5 we discuss computational aspects of the problem. We also give optimal periodic $D$-sequences for most of the $D$-sets proved in this article. 


\section{Distance graphs and Haralambis' Lemma}


The parameters $\kappa(D)$ and $\mu(D)$ are closely related to coloring parameters of distance graphs (cf. \cite{surveyL}).  
Let $D$ be a set of positive integers.  The {\it distance graph generated by $D$}, denoted as $G(\Z, D)$, has the integers $\Z$ as the vertex set. Two vertices are adjacent whenever the absolute value of their difference falls in $D$. Introduced by Eggleton, Erd\H{o}s, and Skilton \cite{realline} in 1985, distance graphs have been studied intensively (cf. \cite{chz,clz,realline, EES86, ees90, KK, KM, KM2, LL,LZ,llz,LL2,L,surveyL,Aileen,lz97,WL,pattern,d=3,zhu03, voigt, vw, vw2}).

The {\it fractional chromatic number} of a graph $G$, denoted by $\chi_f(G)$, is the minimum ratio $m/n$ ($m, n \in Z^+$) of an $(m/n)$-coloring, where an $(m/n)$-coloring is a function on $V(G)$ to $n$-element subsets of $[m]=\{1,2, \cdots, m\}$ such that if $uv \in E(G)$ then $f(u) \cap f(v) = \emptyset$. Chang et al. \cite{clz}  proved that, for any set of positive integers $D$, 
$\chi_f(D) = 1/\mu(D)$. 
Combining this with (1), we have
\begin{equation}
\frac{1}{\mu(D)} = \chi_f(D) \leq \frac{1}{\kappa(D)}. \label{eq:2}
\end{equation}

Viewing the problem of determining $\mu(D)$ through the lens of fractional chromatic number can sometimes 
help to solve open problems  or yield simple proofs (cf. \cite{chz, clz, L, Aileen, lz97, LZ}). 
It is also clear from this point of view that $\mu(D)= \mu(nD)$, where $nD = \{nd \colon d \in D\}$, by considering the fact that $G(\Z, nD)$ 
is just $n$  disjoint, isomorphic copies of $G(\Z, D)$.

The following result, due to Haralambis \cite{H}, is one of the few tools for bounding $\mu(D)$ from above.

\begin{lemma}{\rm \cite{H}}
\label{H}
Let $D$ be a set of positive integers, and let $\alpha \in (0,1]$. If for every $D$-sequence $S$ with $0 \in S$ there exists a positive integer $n$ such that $S[n]/(n+1) \leqslant \alpha$, then $\mu(D) \leqslant \alpha$.
\end{lemma}

We shall frequently use the following equivalent definition of $\kappa(D)$ (cf. \cite{H}). For positive integers $m$ and $x$, denote 
$||x||_m = \min \{x \ ( {\rm mod}\  m {\rm)}, m - x \ ( {\rm mod} \ m {\rm)}\}$.  Let $D$ be a set of positive integers and $t$ a positive integer, then let $||tD||_m = \min \{||td||_m \colon d \in D\}$. With this notation,    
\[
\kappa(D) = \max \bigg\{\frac{ \ ||tD||_m \ }{m} \colon \ \gcd(m, t) = 1 \bigg\}. 
\]

Let $n$ be a positive integer. It can be seen from the definition that $\kappa(nD)=\kappa(D)$. This, together with the fact discussed above that $\mu(D)=\mu(nD)$, allows us to assume gcd$(D)=1$, unless mentioned otherwise. In addition, it can be seen that if $D$ is a singleton or contains only odd numbers, then $\mu(D)=\kappa(D)=1/2$ (cf. \cite{CG}).

%
\section{$D = \{1, j, k\}$ and $j$ is odd}
%

We begin with a result that partially confirms \cref{conj-odd}.

\begin{theorem}
  \label{tj-odd}
Let $D = \{1,j,mj\}$ with $j$ odd, $j \geq 3 $.  Then 
$\mu(D) = \kappa(D) = 1/2$ if $m$ is odd; and 
$\mu(D) = \kappa(D)=\frac{m}{2(m+1)}$ 
if $m$ is even.
\end{theorem}

\begin{proof} 
We only need to show the case when $m$ is even.  
Note, $\mu(D) \leq \mu (\{j, mj\}) = \mu  (\{1,m\}) = \kappa (\{1,m\}) = \frac{m}{2(m+1)}$.  
It is enough to show that there exists some $t$ with $||tD||_{mj+j} \geq mj/2$. One can see that $t=(mj+j-1)/2$ fulfills this requirement.
\end{proof}

In the remainder of this section we show that \cref{conj-odd} does not always hold. \cref{4n+1,4n+3} make use of the following lemma. 

For integers $a \leq b$, denote $\{a,a+1, \ldots, b\}$ as $[a,b]$.  

\begin{lemma}
\label{grid}
Let $D=\{1,j,k\}$ with $j < k$, $j$ odd and $k$ even, and let $S$ be a $D$-sequence. If $| [q(j+1),q(j+1)+j] \cap S | = (j+1)/2$ for some $q$ with $(q+1)(j+1) \leq j+k-1$, then $S[j+k-1] \leq k/2$.
\end{lemma}

\begin{proof}
Partition the elements of $S$ into blocks of length $j+1$:
\[
B_i = S \cap ([0,j]+i(j+1))
\]
for $i \geq 0$. 
Since $1\in D$, $|B_i| \leq (j+1)/2$ for all $i$. 
Let $q$ be the smallest such that $|B_q|=(j+1)/2$, which implies $|B_t| < (j+1)/2$ if $t< q$, and $\{0,2,\dots,j-1\} + (q(j+1)+p) \subset S$ for some $p \in \{0,1\}$. Note that $B_q$ consists of only even numbers when $p=0$ or only odd numbers when $p=1$.  

Let $B^e_i = \{x \in B_i \colon x \equiv p \pmod{2}\}$ and $e_i = |B^e_i|$, and let $B^o_i = \{x \in B_i \colon x \not\equiv p \pmod{2}\}$ and $o_i= |B^o_i|$. Because $q$ is the smallest such that $|B_q|=(j+1)/2$, we know $e_q = (j+1)/2$ and $e_{q-1} < (j+1)/2$. 

\noindent
\textbf{Case 1:} Let $p=0$. We proceed on the following two steps. 

{\bf Step I.} If $q=0$, then let $\lambda = e_0=(j+1)/2$ and move to Step II. Otherwise, since $j$ is odd 
and $e_q = (j+1)/2$, it must be that
$o_q=o_{q-1} = 0$.    
Let $|B_{q-1}| = e_{q-1} = (j+1)/2 - l_{q-1}$ for some $1 \leq l_{q-1} \leq (j+1)/2$. 
Now $o_{q-2} \leq l_{q-1}$, since each number $x \in B^o_{q-2}$ corresponds to $x+j \not\in B^e_{q-1}$. 
Let $|B_{q-2}| = (j+1)/2 - l_{q-2}$, which implies $e_{q-2} \geq (j+1)/2 - l_{q-2}- l_{q-1}$.

By repeating similar arguments we can assume that $|B_{i}| = (j+1)/2 - l_i$ for $0 \leq i \leq q-1$ and that each $e_{i} \geq (j+1)/2- \sum_{n=i}^{q-1} l_n$.
Letting $\sum_{n=0}^{q-1} l_n = l$ we have $e_o \geq (j+1)/2 - l$ and 
\begin{align*}
S[j+k-1] &\leq \sum_{i=0}^{q}|B_i| + \frac{j+k+1 - (q+1)(j+1)}{2}\\
         &= \frac{(q+1)(j+1)}{2} - l + \frac{j+k+1 - (q+1)(j+1)}{2}\\
         &= \frac{k}{2} + \frac{j+1}{2} - l. 
\end{align*}

\noindent
Hence, the result follows if $l \geq (j+1)/2$. If not, let $\lambda = (j+1)/2 - l$.  We can assume that $e_0 \geq \lambda \geq 1$. To prove our lemma, it suffices to show that 
\[
|S \cap [(q+1)(j+1), j+k-1]| \leq \frac{j+k+1 - (q+1)(j+1)}{2} - \lambda.
\]

{\bf Step II.} Let $k = m(j+1) + \overline{k}$ for some $0\leq \overline{k} \leq j$. Partition $S \cap [0,j+k-1]$, {\it backwards} instead of forwards:  
\[
B'_i = S \cap ([\overline{k}-1, \overline{k}+j-1] + i(j+1)) 
\]
for $0 \leq i \leq m$. Let $e'_i = |\{x \in B'_i \colon x \equiv p \pmod{2}\}|$ and $o'_i = |\{x \in B'_i \colon x \not\equiv p \pmod{2}\}|$. 
Note that $B'_m = S \cap [k-1, k+j-1]$.  

Since $e_0 \geq \lambda$ and $k \in D$, $e'_{m} \leq (j+1)/2 - \lambda$.  Thus let $|B'_m| = (j+1)/2 - \lambda_{m}$ for some $0\leq \lambda_{m} \leq \lambda$ with $o'_{m} \geq \lambda-\lambda_{m}$. 

By similar arguments we can assume that $|B'_i| = (j+1)/2 - \lambda_i$ and that $o'_i \geq \lambda - \sum_{t=i}^{m} \lambda_t$ for $q \leq i \leq m$. Considering that $B'_{q} = S\cap [q(j+1) +\overline{k} -1, (q+1)(j+1)+\overline{k}-2]$ 
and that 
$o_{q-1} = o_q = 0$ 
by hypothesis, it must be that $o'_{q} = 0$. Thus $\sum_{i=q}^{m} \lambda_i \geq \lambda$. Noting that if $\overline{k} \leq 1$, then $\lambda_q = 0$, and if $\overline{k} \geq 2$, then the excluded numbers counted by $\lambda_q$ all are greater than or equal to $(q+1)(j+1)$, we obtain 
\begin{align*}
|S \cap [(q+1)(j+1), j+k-1]|  
&= \frac{j+k+1 - (q+1)(j+1)}{2}  - \sum_{i=q}^{m} \lambda_i \\ 
&\leq \frac{j+k+1 - (q+1)(j+1)}{2}  - \lambda. 
\end{align*}
\noindent
\textbf{Case 2:} Let $p=1$.

\textbf{Step I.} This step is very similar to the first case. The only difference is that since $B_q$ contains all odd numbers, it might be that $(q-1)(j+1) \in B_{q-1}$ and similarly each $o_i \leq l_{i+1} +1$. 
Thus, letting $|B_i|=(j+1)/ 2 -l_i$ and 
$l = \sum_{n=0}^{q-1} l_n$, we conclude that $e_o \geq (j+1)/2 - l-1 = (j-1)/2-l$.

Since $(q+1)(j+1)-1 \in B_q$, we know $(q+1)(j+1) \not\in S$. Thus
$|S \cap [(q+1)(j+1), j+k-1]| \leq (j+k-1 - (q+1)(j+1))/2$ and therefore 
\begin{align*}
S[j+k-1] &\leq \frac{(q+1)(j+1)}{2} - l + \frac{j+k-1 - (q+1)(j+1)}{2}\\
         &= \frac{k}{2} + \frac{j-1}{2} - l. 
\end{align*}
Let $\lambda = (j-1)/2-l$.  
To complete the proof it suffices to show that
\[
|S \cap [(q+1)(j+1), j+k-1]| \leq \frac{j+k-1 - (q+1)(j+1)}{2} - \lambda.
\]

\textbf{Step II.} This step follows Case 1.
\end{proof}

We now prove \cref{4n+1,4n+3} in the Introduction, and restate them below.  

\medskip
\noindent\textbf{\cref{4n+1}.}
\textit{
Let $D = \{1,j,k\}$ with $j=4n+1$, $k$ even, and $k \leq j(j-1)/2$. 
If $k=m(j+1) + \overline{k}$ with $\overline{k}$ as described below, then 
\[
\kappa(D) = \mu(D) = 
\begin{cases}
\frac{n}{2n+1}      & \textit{if }\overline{k} = 2n, 2n+2\\
\frac{k-2m}{2(k+1)} & \textit{if } \max\{0, 2(n-m)\} \leq \overline{k} \leq 2(n-1).
\end{cases}
\]
}

\begin{proof}
\noindent 
\textbf{Case 1:} $\overline{k} = 2n$ or $\overline{k} = 2n+2$. 
Since $k \equiv \pm 1 \pmod {2n+1}$ and $j \equiv -1$ (mod $2n+1$), one can easily check that $||nD||_{2n+1} \geq n$.  Thus $\kappa(D) \geq n/(2n+1)$. It remains to show $\mu(D) \leq n/(2n+1)$. We proceed by considering two subcases.

\noindent
{\bf Case 1.1:} $\overline{k} = 2n$.  Then $k=(2m+1)(2n+1) -1$, where $1\leq m \leq 2n-1$ since $k \leq j(j-1)$. 
If $S[j] \leq 2n$ for every $D$-sequence $S$ with $0\in S$, then the results follows by \cref{H}. 
Thus assume there exists a $D$-sequence $S$ with $0 \in S$ and $S[j]=2n+1$. 
This implies $\{0,2,\dots,j-1\} \subset S$. 
By \cref{H} it is enough to show that 
$
S[k+j-1] \leq 2mn+3n-1. 
$
Because $m\leq 2n-1$, \cref{grid} implies      
$
S[j+k-1] \leq 2mn  +3n-1.  
$

\noindent
{\bf Case 1.2:} $\overline{k} = 2n+2$. Then $k = (2m+1)(2n+1) +1$, where $0 \leq m \leq 2n-1$. 
Let $S$ be a $D$-sequence with $0 \in S$. If $S[j] \leq 2n$, by \cref{H} the result follows. So we assume $S[j] \geq 2n+1$, which implies  $\{0, 2, \ldots, j-1\} \subset S$.  
Since $m \leq 2n-1$, by \cref{grid}, we have 
$
S[j+k-1] \leq 2mn  +3n. 
$
Again, the result follows by \cref{H}.

\noindent
\textbf{Case 2:} $\max\{0, 2(n-m)\} \leq \overline{k}\leq 2(n-1)$. Since $(k/2 - m)k = (k/2-m)(k+1) - (k/2-m)$ and with the assumption on $\overline{k}$ it can be verified that
\[
(2n-1)(k+1)+(k/2-m) \leq (4n+1)(k/2-m) \leq 2n(k+1)-(k/2-m). 
\] 
This implies 
$||(k/2-m)D||_{k+1} \geq k/2-m$ and thus $\kappa(D)\geq (k/2-m)/(k+1)$.

Next, we claim that $\mu(D) \leq \frac{k-2m}{2(k+1)}$. The assumption $k \leq j(j-1)/2$ yields that $m \leq 2n-1$. 
Let $S$ be a $D$-sequence with $0 \in S$.  If $|S \cap [i(j+1), i(j+1)+j]| = (j+1)/2$ for any $0 \leq i \leq (m-1)$, then by \cref{grid}, $S[j+k-1] \leq k/2$. This implies 
$
S[j+k-1]/(j+k) \leq \frac{k-2m}{2(k+1)}.   
$
By \cref{H} the result follows. 

Thus, assume $|S \cap [i(j+1), i(j+1)+j]| \leq (j-1)/2$ for all $0 \leq i \leq (m-1)$.  
Summing up, 
$S[m(j+1)-1] \leq m [ (j-1)/2 ]$. Hence, $S[k] \leq m [ (j-1)/2 ] + \overline{k}/2 = (k-2m)/2$.  Again, the result follows by \cref{H}.  
\end{proof}

Note, \cref{4n+1} implies that the boundary condition $k \geq j(j-1)/2$ in \cref{H-odd} is sharp. When $k=j(j-1)/2=8n^2+2n$ (that is, $m=2n-1$ and $\overline{k}=2n+2$), \cref{H-odd,4n+1} produce the same value of $\mu(D)$, while this is not the case when $k \neq j(j-1)/2$.

\medskip
\noindent\textbf{\cref{4n+3}.}
\textit{
Let $D = \{1,j,k\}$ with $j=4n+3$, $k$ even, and $k \leq j(j-1)/2$.
If $k=m(j+1) + \overline{k}$ with $\overline{k}$ as described below, then
\[
\kappa(D) = \mu(D) = 
\begin{cases}
\frac{n}{2n+1}      & \textit{if }\overline{k} = 2(n-m) \  \textit{and} \ m \leq 2n\\
\frac{k-2m}{2(k+1)} & \textit{if } \max\{0, 2(n-m+1)\} \leq \overline{k} \leq 2n \\
\frac{2n+1}{4n+4}   & \textit{if }\overline{k} = 2(n+1).
\end{cases}
\]
}

\begin{proof}
\noindent
\textbf{Case 1:} Let $\overline{k} = 2(n-m)$. 
Then $k = 2m(2n+1)+2n$, where $m \leq 2n$. 

First we claim $\kappa(D) \geq n/(2n+1)$. By hypotheses, we have $||k||_{2n+1} = ||j||_{2n+1}=1$, so $||nD||_{2n+1} = n$, and $\kappa(D) \geq n/(2n+1)$.

Next we prove $\mu(D) \leq n/(2n+1)$. By inequality (2), it is sufficient to show that $\chi_f (D) \geq (2n+1)/n$ by exhibiting a cycle of length $2n+1$ as a subgraph in $G(\Z, D)$ since $\chi_f(C_{2n+1})=(2n+1)/n$. As $m \leq 2n$, the following vertices form a cycle of length $2n+1$ in $G(\Z, D)$: 
$
\{0, j, 2j, \ldots, mj, mj+1, mj+2, \ldots, mj+2n-m=k\}. 
$

\noindent
\textbf{Case 2:} $\overline{k} = 2(n+1)$. Then $k=m(4n+4)+(2n+2)$, where $m \leq 2n$ as $k \leq j(j-1)/2$. 

Since $k \equiv 2n+2$ (mod $4n+4$) and $j \equiv -1$ (mod $4n+4$), one can easily check that $||(2n+1)D||_{4n+4} \geq 2n+1$.  Thus $\kappa(D) \geq (2n+1)/(4n+4)$.

Next we show $\mu(D) \leq (2n+1)/(4n+4)$.  
Let $S$ be a $D$-sequence with $0 \in S$.  
If $S[j] \leq 2n+1$,  by \cref{H} the result follows.  Thus assume 
$S[j] \geq 2n+2$, which implies that  $\{0,2,\dots,j-1\} \subset S$.
Because $m \leq 2n$, by \cref{grid}, we have 
$
S[j+k-1] \leq 2nm + 2m +n +1 \leq 2nm + m +3n+1. 
$
The result follows by \cref{H}.

\noindent
\textbf{Case 3:} $\max\{2(n-m+1), 0\} \leq \overline{k} \leq 2n$. 

Since $(k/2 - m)k = (k/2-m)(k+1) - (k/2-m)$ and with the assumptions on $\overline{k}$ it can be verified that 
\[
(2n)(k+1)+(k/2-m) \leq (4n+3)(k/2-m) \leq (2n+1)(k+1)-(k/2-m), 
\]
we obtain $||(k/2-m)D||_{k+1} \geq k/2-m$.  Thus $\kappa(D)\geq (k/2-m)/(k+1)$.

The proof that $\mu(D) \leq \frac{k-2m}{2(k+1)}$ is similar to \cref{4n+1} Case 2.
The assumption $k \leq j(j-1)$ yields that $m \leq 2n$.
Let $S$ be a $D$-sequence with $0 \in S$.  If $|S \cap [i(j+1), i(j+1)+1]| = (j+1)/2$ for some $0 \leq i \leq (m-1)$, then by \cref{grid}, $S[j+k-1] \leq k/2$, which  implies 
$
S[j+k-1]/(j+k) \leq \frac{k-2m}{2(k+1)}.   
$

Thus assume $|S \cap [i(j+1), i(j+1)+j]| \leq (j-1)/2$ for every $0 \leq i \leq (m-1)$.  Summing up, 
$S[m(j+1)-1] \leq m [ (j-1)/2 ]$. Hence, $S[k] \leq m [ (j-1)/2 ] + \overline{k}/2 = (k-2m)/2$.  By \cref{H}, the result follows.    
\end{proof}

Similar to \cref{4n+1}, \cref{4n+3} implies that the assumption $k \geq j(j-1)/2$ in \cref{H-odd} is sharp.

%
\section{$D = \{1, j, k\}$ and $j$ is even}
%
\cref{H-even} applies to most $D$-sets in this family, in the sense that for a given even integer $j$ there are only a finite number of integers $k$ such that $D=\{1,j,k\}$ is not covered by \cref{H-even}. But the  restrictions of \cref{H-even} on $k$ are not simple bounds, so those $k$ for which $\mu(D)$ is still unknown are not all consecutive. Considering the case when $k$ is a multiple of $j$, direct calculation shows that \cref{H-even} determines the values of $\mu(D)$ when $k$ is an odd  multiple of $j$.

Next we determine the value of $\mu(D)$ when $k=mj$ and $m$ is even. First, we get the upper bound for $\mu(D)$, $\mu(\{1,j,mj\}) \leq \mu (\{j,mj\}) = \frac{m}{2(m+1)}$. To establish the lower bound for  $\kappa(D)$ the ``good time'' $t$ approach used in the proof of \cref{tj-odd} no longer works. Instead, we find the kappa value by using the following lemma which holds for general three-element $D$-sets.

\begin{lemma}
\label{kappa-lemma}
Let $D = \{i,j,k\}$, where $\gcd (D)=1$ and $\gcd (i,j)=d$. If $\frac{i}{d} + \frac{j}{d} = 2x+1$ and $d \geq 2x$, then $\kappa (D) = \frac{x}{2x+1}$.
\end{lemma}

\begin{proof}
First note that $\frac{x}{2x+1} = \frac{(i+j-d)/2}{i+j} = \kappa(\{i,j\})$ as proved in [4]. So it remains to show that $\kappa(D) \geq \frac{x}{2x+1}$.

The hypotheses imply that $2x+1$ and $i/d$ are relatively prime, so let $a$ be the unique solution to the congruence equation
$
a(i/d) \equiv x \pmod {2x+1}
$
such that $0\leq a \leq 2x$.
Multiplying this equation by $d$ gives
$
ai \equiv (i+j-d)/2 \pmod {i+j}.
$
This together with the fact that $j \equiv -i \pmod {i+j}$ implies that
$
aj \equiv (i+j+d)/2 \pmod {i+j}.
$
As $(2x+1)i \equiv (2x+1)j \equiv 0$ (mod $i+j$), the following holds 
for every positive integer $n$
\[
||(a+n(2x+1))i||_{i+j} =  ||(a+n(2x+1))j||_{i+j}  
= (i+j-d)/2.
\]
Thus, it suffices to show there exists  some $n$ such that $||(a+n(2x+1))k ||_{i+j} \geq \frac{i+j-d}{2}$.

Let $ak\equiv c \pmod{2x+1}$ for some $0\leq c \leq 2x$. Clearly $(a+n(2x+1))k\equiv c \pmod{2x+1}$ for all $n$. The residue classes of $(a+n(2x+1))k$ modulo $i+j$ will range over the entire set $\{c+m(2x+1) \mid 0 \leq m \leq d-1\}$ as $n$ ranges from $0$ to $d-1$. To show this, assume $p$ and $q$ are distinct integers with $0\leq p \leq q \leq d-1$ and 
\[
k(a+p(2x+1)) \equiv k(a+q(2x+1)) \pmod{i+j}.
\]
Then, $k(q-p) = ld$ for some positive integer $l$. This is impossible as $(q-p) < d$ and gcd$(k,d)=1$. 

Because $\frac{i+j+d}{2} -\frac{i+j-d}{2} = d$ and by the hypothesis $d \geq 2x$, there exists some $m \in \{0,1,\dots,d-1\}$ such that $\frac{i+j-d}{2} \leq c+m(2x+1) \leq \frac{i+j+d}{2} $, as needed. Thus, $\kappa(D) \geq x/(2x+1)$. 
\end{proof}

\begin{theorem}
\label{tj-even}
Let $D = \{1,j,mj\}$ with both $j$ and $m$ even and $m \leq j$. Then $\mu(D) = \kappa(D) = \frac{m}{2(m+1)}$.
\end{theorem}

\begin{proof}
First, $\mu(D) \leq \mu (\{j, mj\}) = \mu (\{1,m\}) = \kappa (\{1,m\}) = \frac{m}{2(m+1)}$.
When $m\leq j$, by \cref{kappa-lemma},  $\mu(D) \geq \kappa(D) = \frac{m}{2(m+1)}$.
\end{proof}

Note, when $k=mj$ and both $m$ and $j$ are even, \cref{H-even} applies for $m \geq j-2$.  Thus, \cref{H-even,tj-even} together settle the family $D=\{1,j,k\}$ where $k$ is an even multiple of $j$. When $m \in \{j-2,j\}$, both \cref{H-even,tj-even} apply, but when $m \leq j-4$ \cref{tj-even} gives a different value than \cref{H-even} would without the boundary condition on $k$. Thus, the boundary condition in \cref{H-even} is sharp when $k$ is a multiple of $j$.

Recall \cref{new} form the Introduction, which shows that the boundary condition in \cref{H-even} is also sharp when $k \equiv 3$ (mod $j+1$).

\medskip
\noindent\textbf{\cref{new}.}
\textit{Let $D=\{1, j, k\}$, where $j$ is even, $k=m(j+1)+3$, and $1 \leq m \leq j-3$.  
Then
\[
\mu(D) = \kappa(D) =\frac{j(m+1)}{2(j+k)}. 
\]}
 
\begin{proof}
Let $t=(j/2)(m+1)+1$. Since $tj= (j/2)(j+k) -(j/2)(m+1)$ and $tk= (mj/2 +1)(j+k) + (j/2)(m+1)$, we conclude that $||tD||_{j+k} =(j/2)(m+1)$ and $\mu(D) \geq \kappa(D) \geq \frac{j(m+1)}{2(j+k)}$.

To show that $\mu(D) \leq \frac{j(m+1)}{2(j+k)}$, let $S$ be a $D$-sequence with $0 \in S$. Partition $S \cap [0,k]$ into sets $A_i = S \cap ([0,j] + i(j+1))$ 
for $0\leq i \leq m-1$ and $B = S \cap [m(j+1), k]$. Note that $|A_i| \leq j/2$ for all $i$, and since $k\in D$ and $0\in S$, $k \not\in B$. If $S[k] < (mj/2)+2$, then by \cref{H} we are done, thus we can assume that $|A_i| = j/2$ for all $i$ and that $B=\{m(j+1),m(j+1)+2\} =\{k-3, k-1\}$. 
Considering the structure of $B$ we know that 
$(\{1,3,j\}+(m-1)(j+1))\cap A_{m-1} = \emptyset$. 
We also know that $\{0,2\}+(m-1)(j+1) \subset A_{m-1}$, since otherwise $|A_{m-1}| < j/2$. Together these facts determine the structure of $A_{m-1}$:
$$
A_{m-1} = \{0,2,\dots,e_1-2, e_1+1, e_1+3,\dots, j-1\} + (n-1)(j+1)
$$
where $e_1$ is even and $e_1 \geq 4$. Applying similar arguments to $A_{m-2},A_{m-3},\dots,A_0$, we obtain
$$
A_0 = \{0,2,\dots,e_m-2, e_m+1, e_m+3,\dots, j-1\}
$$
where $e_m$ is even and $e_m\geq e_{m-1} \geq \cdots \geq e_1 \geq 4$.

Now let $A_m = S \cap [m(j+1),(m+1)(j+1)-1]$ and $C = S\cap \{j+k-2, j+k-1\}$. If $e_m = j$, that is $j-2 \in A_0$, then $|C| = 0$ and $S[j+k-1] \leq (m+1)j/2$, and the result follows from \cref{H}. Thus we can assume that $4 \leq e_1 \leq e_m \leq j-2$. The fact that $(m-1)(j+1) + e_1 +1, (m-1)(j+1) + e_1 +3 \in S$  (note, if $e_1=j-2$, then $(m-1)(j+1) + e_1 +3=m(j+1)$ which is in $B$ as discussed above) 
forces $m(j+1)+e_1, m(j+1)+e_1 +2 \not\in A_m$, and the fact that $e_1-2 \in S$ forces $m(j+1)+e_1+1\not\in A_m$. So $|A_m| \leq j/2 -1$ and $S[j+k-1] \leq (m+1)j/2$, and the result follows from \cref{H}.
\end{proof}

%
\section{Optimal sequences and fractional coloring}
%

While we have proved the exact value of $\mu(D)$ for several families of $D$-sets, none of these proofs are constructive. Examining the form of maximally dense $D$-sequences gives another interesting perspective to the problem. We shall write the elements of a $D$-sequence $S$ in an increasing order, $S = s_0, s_1, s_2, \ldots$ with $s_0 < s_1 < s_2 < \ldots$, and denote the {\it difference sequence} of $S$ by $\Delta(S) = \delta_0, \delta_1, \delta_2, \ldots$ where $\delta_i = s_{i+1} - s_i$. A subsequence of consecutive terms in $\Delta(S)$, $\delta_a, \delta_{a+1}, \ldots, \delta_{a+b-1}$, generate a periodic interval of $k$ copies, $k \geq 1$, if $\delta_{j(a+b)+i} = \delta_{a+i}$ for all $0 \leq i \leq b-1$, $1 \leq j \leq k-1$. Denote such a periodic subsequence by $(\delta_a, \delta_{a+1}, \ldots, \delta_{a+b-1})^k$. If the periodic interval repeats infinitely we simply denote it by $(\delta_a, \delta_{a+1}, \ldots, \delta_{a+b-1})$. One can easily calculate the density of $S$ from $\Delta(S)$, e.g., if $\Delta(S)=((a,b)^2,c^3,e)$, then $\delta(S) = 8/(2(a+b)+3c+e)$. An {\it optimal sequence} has density equal to $\mu(D)$.

From inequality (1), one can use $\kappa(D)$ to provide a lower bound for $\mu(D)$. Tabulating these values for $D$-sets is an effective way to discover patterns. See \cref{kappa_j_odd,kappa_j_even} in the Appendix. The value of $\kappa(D)$ can be computed using the following algorithm, which takes advantage of the fact that $\kappa(D)$ must have the sum of two elements of the $D$-set as denominator. The second procedure of the algorithm, \textsc{D-Seq}, finds the associated $D$-sequence and is essentially the computational version of Cantor and Gordon's proof that $\kappa(D) \leq \mu(D)$ \cite{CG}.  
The idea is to use the ``good time'' $t$ associated with the calculated value of $\kappa(D)=c/m$. 
Since $||td||_{m} \geq c$ for all $d \in D$, the set 
$\cup_{r=0}^{c-1}\{n \colon tn \equiv r \pmod{m}\}$
forms a periodic $D$-sequence with density $c/m$. 
Note that \textsc{D-Seq} works with any input values for circumference $m$ and time $t$ so long as they are relatively prime. Furthermore,   
\textsc{D-Seq} produces an optimal $D$-sequence if $\kappa(D) = \mu(D)$. 

\begin{algorithm}[H]
\caption{Find kappa value of a given $D$-set and associated $D$-sequence}
\label{alg:kappa}
\begin{algorithmic}[1]
\Procedure{kappa}{D}
\State $K \gets \emptyset$ \Comment{Intialize set of possible kappa values to maximize}
\State circumferences $\gets \{a+b \colon a \in D \text{ and } b \in D \text{ and } a \not= b\}$
\For{$m$ in circumferences}
	\For{time $t \leq m/2$}
    	\State minDist $\gets \min\{|td|_m \colon d \in D\}$
        \State Append the ratio minDist/$m$ to $K$
    \EndFor
\EndFor
\State \textbf{return} $\max(K)$ \Comment{The kappa value for $D$}
\EndProcedure
\Procedure{D-Seq}{$D$, $m$, $t$} \Comment{$\gcd(m,t) = 1$}
\State $S \gets \emptyset$ \Comment{Intialize $D$-sequence}
\State minDist $\gets \min\{|td|_m \colon d \in D\}$
\For{$0\leq r \leq \text{minDist} - 1$}
	\State $n \gets$ the solution to the eqn $r \equiv tn \pmod{m}$ such that $0 \leq n \leq m-1$
    \State Append $n$ to $S$
\EndFor
\State \textbf{return} Sort $S$ \Comment{S is a periodic subsequence that generates a $D$-sequence}
\EndProcedure
\end{algorithmic}
\end{algorithm}

Following is a table of optimal difference sequences for some of the $D$-sets studied in this paper. 
\Cref{diffseq_end} in the Appendix 
shows some particular examples.

\begin{table}[H] \centering
\caption{Optimal Difference Sequences for some $D= \{1,j,k\}$}
\label{diffseq}
\begin{tabular}{llll} \toprule
$j$     & $k$                & $\Delta(S)$                     & Theorem             \\ \midrule
$2n+1$  & $mj,$ \ \text{$m$ even}               & $(2^{mj/2 -1}, j+2)$            & \cref{tj-odd}        \\
$4n+1$  & $m(j+1) + 2n$      & $(2^{n-1},3)$                   & \cref{4n+1}, Case 1   \\
$4n+1$  & $m(j+1) + 2n+2$    & $(2^{n-1},3)$                   & \cref{4n+1}, Case 1   \\
$4n+3$  & $m(j+1) + 2(n-m)$  & $(2^{n-1},3)$                   & \cref{4n+3}, Case 1   \\
$4n+3$  & $m(j+1) + 2(n+1)$  & $(2^{n},3,2^{n-1},3)$           & \cref{4n+3}, Case 2   \\
$2n$    & $m(j+1) + 3$       & $(2^{j/2-1}, 3)^m(2^{j/2-1}, 5)$ & \cref{new}           \\
 \bottomrule
\end{tabular}
\end{table}

There is a natural correspondence between 
a periodic $D$-sequence and a certain class of periodic fractional colorings of the distance graph $G(\Z, D)$.
\begin{prop}
Let $D$ be a set of nonnegative integers. Then there is a periodic $D$-sequence $S$ with $\Delta(S) = (\delta_0, \delta_1, \dots, \delta_{n-1})$ and $\sum_{i=0}^{n-1}\delta_i = m$ if and only if there is a periodic $(m/n)$-coloring of $G(\Z,D)$ defined for each $v$ in the vertex set of $G(\Z, D)$ by $v \mapsto \{v+\delta_0, v+\delta_0 + \delta_1, \dots, v + m\}$ where addition is carried out modulo $m$.
\end{prop}

\begin{proof}
Fixing a $D$-set, let $S$ be a sequence with $\Delta(S) = (\delta_0, \delta_1, \dots, \delta_{n-1})$ and $\sum_{i=0}^{n-1}\delta_i = m$. By definition, $S$ is a $D$-sequence if and only if $\sum\limits_{i=a}^{b} \delta_i \not\in D$ for all $a\leq b$, which is equivalent to $
\delta_a + \delta_{a+1}+ \cdots + \delta_{b} \not\equiv d  \pmod{m}
$
holding true for all $d \in D$. Replacing $d$ with any two adjacent vertices $u$ and $v$ in the distance graph where $v<u$, we get
\begin{align*}
\delta_a + \delta_{a+1}+ \cdots + \delta_{b} &\not\equiv u-v &\pmod{m} \\
v + \delta_0 + \cdots +\delta_b &\not\equiv u + \delta_0 + \cdots +\delta_{a-1} &\pmod{m}.
\end{align*}
This holds if and only if for any adjacent vertices $u$ and $v$ the sets $\{u+\delta_0, u+\delta_0 + \delta_1, \dots, u + m\}$ and $\{v+\delta_0, v+\delta_0 + \delta_1, \dots, v + m\}$ are disjoint and thus if and only if the $(m/n)$-coloring defined in the statement of the lemma is proper.
\end{proof}

\section{Conclusion}

Directly computing the value of $\mu(D)$ and finding an optimal periodic sequence is a difficult task in general, while 
finding a lower bound is relatively easy, either by using the kappa value or finding a dense $D$-sequence. In the face of such difficulty, a conjecture of Haralambis is particularly interesting.
\begin{guess} {\rm \cite{H}}
\label{h-conj}
Every 3-element $D$-set has $\kappa(D) = \mu(D)$.
\end{guess}

Computer calculations have confirmed \cref{h-conj} for $\max(D) \leq 25$. 
For $|D| \geq 4$, there exist 4-element $D$-sets with $\mu(D) > \kappa(D)$ (cf. \cite {H, LZ})  that can be extended to $D$-sets of any cardinality greater than 4 which maintain this strict inequality.

As an example of how the kappa value can be more easily dealt with, we end with a final proposition.

\begin{prop}
\label{last}
Let $D = \{1,j,k\}$ with $j = 2n+1$ and $k = m(j+1) + \overline{k}$. For any positive integer $a$, 
\[
\kappa(D) \geq
\begin{cases}
\frac{k/2 - am}{k+1}  & \textit{ if } \frac{n-a-2am}{a} \leq \overline{k} \leq \frac{n-a+1}{a} \\
\frac{k/2 -am - a +1}{k+1} & \textit{ if } \frac{2an - n -a -2am}{a} \leq \overline{k} \leq \frac{2an-n+a-1}{a}.
\end{cases}
\]
\end{prop}

\begin{proof}
\textbf{Case 1.} Let $\frac{n-a-2am}{a} \leq \overline{k} \leq \frac{n-a+1}{a}$.
\medskip

Since \( (k/2 -am)k = (k/2 -am)(k+1) - (k/2 - am) \), and with the assumption on $\overline{k}$ it can be verified that
\[
(n-a)(k+1) + (k/2 - am) \leq (k/2 - am)j \leq (n-a+1)(k+1) - (k/2 - am),
\]
we obtain that $||(k/2 - am)D||_{k+1} \geq k/2 - am$, and thus $\kappa(D) \geq \frac{k/2 - am}{k+1}$. 
\medskip

\noindent\textbf{Case 2.} Let $\frac{2an - n -a -2am}{a} \leq \overline{k} \leq \frac{2an-n+a-1}{a}$.
\medskip

This case is similar to Case 1, since it can be verified that
\[
(n-a)(k+1) + (k/2 - am-a+1) \leq (k/2 - am-a+1)j \leq (n-a+1)(k+1) - (k/2 - am-a+1).
\]
\end{proof}

\noindent
When $a=1$ this formula becomes the same as parts of \cref{4n+1,4n+3}. Recall that the denominator of the kappa value of a $D$-set must be the sum of two elements in $D$.  Hence the three possible denominators of the kappa value for $D=\{1,j,k\}$ are $k+1$, $j+1$, and $k+j$. All computer generated data show that equality holds for \cref{last} for all $D$-sets of the form $\{1,j,k\}$ with $j$ odd that have $k+1$ as the denominator of their kappa value. However there are some $D$-sets satisfying the conditions of the proposition which have a kappa value strictly greater than the bound in \cref{last}, but these kappa values have $j+k$ or $j+1$ as the denominator.




\newpage
\appendix 
\section{Tables}
\normalsize
\begin{table}[H]
\caption{Optimal Difference Sequences for some $D= \{1,j,k\}$}
\label{diffseq_end}
\centering
\begin{tabular}{llll} \toprule
$j$     & $k$       & $\Delta(S)$                                     & Theorem      \\ \midrule
$5$     & $6$       & $(3,4)$                                         & \cref{4n+1}, Case 2  \\
$7$     & $10$      & $(2,3,3,3) = (2,3^3)$                                     & \cref{4n+3}, Case 3  \\
$7$     & $16$      & $(2,3,3,3,3,3)=(2,3^5)$                                 & \cref{4n+3}, Case 3  \\
$7$   & $18$     & $(2,3,3,2,3,3,3)=((2,3,3)^2,3)$                             & \cref{4n+3}, Case 3  \\
$9$     & $12$      & $(2,3,2,3,3) = ((2,3)^2,3)$                                   & \cref{4n+1}, Case 2  \\
$9$     & $20$      & $(2,3,2,3,3,2,3,3)= ((2,3),(2,3,3)^2)$                             & \cref{4n+1}, Case 2  \\
$9$     & $22$      & $(2,3,2,3,2,3,2,3,3) = ((2,3)^4,3)$                           & \cref{4n+1}, Case 2  \\
$9$     & $30$      & $((2,3)^3,3,(2,3)^2,3)$                     & \cref{4n+1}, Case 2  \\
$9$     & $32$      & $((2,3)^6,3)$                   & \cref{4n+1}, Case 2  \\
$11$    & $16$      & $(2,(2,3)^3)$                               & \cref{4n+3}, Case 3  \\
$11$    & $26$      & $(2,(2,3)^5)$                       & \cref{4n+3}, Case 3  \\
$11$    & $28$      & $(2,(2,3)^3,2, (2,3)^2)$                     & \cref{4n+3}, Case 3  \\
$11$    & $36$      & $(2,(2,3)^7)$               & \cref{4n+3}, Case 3  \\
$11$    & $38$      & $(2,(2,3)^4,2,(2,3)^3)$             & \cref{4n+3}, Case 3  \\
$11$    & $40$      & $(2,(2,3)^3,2,(2,3)^2,2,(2,3)^2)$           & \cref{4n+3}, Case 3  \\
$11$    & $48$      & $(2,(2,3)^5,2,(2,3)^4)$     & \cref{4n+3}, Case 3  \\
$11$    & $50$      & $(2,(2,3)^3)$                               & \cref{4n+3}, Case 3  \\
$11$    & $52$      & $(2,(2,3)^3,(2,(2,3)^2)^3$ & \cref{4n+3}, Case 3  \\
$13$    & $18$      & $((2,2,3)^2,2,3)$                             & \cref{4n+1}, Case 2  \\
$13$    & $30$      & $((2,2,3)^2,2,3,2,2,3,2,3)$                   & \cref{4n+1}, Case 2  \\
$13$    & $32$      & $((2,2,3)^4,2,3)$                 & \cref{4n+1}, Case 2  \\
$13$    & $42$      & $((2,2,3)^2,2,3,(2,2,3,2,3)^2)$         & \cref{4n+1}, Case 2  \\
$13$    & $44$      & $((2,2,3)^3,2,3,(2,2,3)^2,2,3)$       & \cref{4n+1}, Case 2  \\
$13$    & $46$      & $((2,2,3)^6,2,3)$     & \cref{4n+1}, Case 2  \\
$13$    & $56$      & $((2,2,3)^2,2,3)$                                  & \cref{4n+1}, Case 2  \\
$13$    & $58$      & $((2,2,3)^4,2,3,(2,2,3)^3,2,3)$  & \cref{4n+1}, Case 2  \\
$13$    & $60$      & $((2,2,3)^8,2,3)$  & \cref{4n+1}, Case 2  \\
$13$    & $70$      & $((2,2,3)^3,2,3,(2,2,3)^3,2,3,(2,2,3)^2,2,3)$  & \cref{4n+1}, Case 2  \\
$13$    & $72$      & $((2,2,3)^5,2,3,(2,2,3)^4,2,3)$  & \cref{4n+1}, Case 2  \\
$13$    & $74$      & $((2,2,3)^{10},2,3)$  & \cref{4n+1}, Case 2  \\
 \bottomrule
 \end{tabular}
\end{table}

\begin{landscape}
\begin{table} 
\caption{$\kappa(\{1,j,k\})$ when $j$ is odd and $k$ is even}
\label{kappa_j_odd}
\centering \small
\begin{tabular}{*{16}{r}} \toprule
k/j&3&5&7&9&11&13&15&17&19&21&23&25&27&29&31\\ \midrule
2&1/4&1/3&1/3&3/10&1/3&1/3&5/16&1/3&1/3&7/22&1/3&1/3&9/28&1/3&1/3\\
4&2/7&1/3&3/8&2/5&2/5&5/14&3/8&7/18&2/5&2/5&3/8&5/13&11/28&2/5&2/5\\
6&1/3&2/7&4/13&2/5&5/12&3/7&3/7&9/23&2/5&9/22&5/12&11/26&3/7&3/7&13/32\\
8&4/11&1/3&1/3&3/10&7/19&3/7&7/16&4/9&4/9&12/29&12/31&13/33&3/7&13/30&7/16\\
10&5/13&1/3&4/11&6/19&1/3&5/14&2/5&4/9&9/20&5/11&5/11&3/7&15/37&5/13&16/41\\
12&2/5&6/17&3/8&5/13&4/13&8/25&3/8&5/13&13/31&5/11&11/24&6/13&6/13&18/41&18/43\\
14&7/17&7/19&1/3&2/5&2/5&1/3&5/16&7/18&2/5&2/5&16/37&6/13&13/28&7/15&7/15\\
16&8/19&8/21&6/17&2/5&7/17&11/29&10/31&1/3&2/5&9/22&7/17&17/41&19/43&7/15&15/32\\
18&3/7&9/23&7/19&1/3&5/12&8/19&7/19&6/19&12/37&5/13&5/12&8/19&2/5&20/47&22/49\\
20&10/23&2/5&3/8&8/21&12/31&3/7&3/7&8/21&1/3&7/22&3/8&11/26&3/7&3/7&7/17\\
22&11/25&11/27&11/29&9/23&1/3&3/7&10/23&16/39&9/23&14/43&1/3&5/13&3/7&13/30&10/23\\
24&4/9&12/29&12/31&2/5&2/5&14/37&7/16&11/25&2/5&2/5&8/25&16/49&11/28&22/53&7/16\\
26&13/29&13/31&13/33&2/5&11/27&1/3&17/41&4/9&4/9&11/27&19/49&1/3&9/28&2/5&13/32\\
28&14/31&14/33&2/5&11/29&12/29&16/41&3/8&4/9&13/29&3/7&12/29&11/29&18/55&1/3&23/59\\
30&5/11&3/7&15/37&12/31&5/12&13/31&1/3&19/47&9/20&14/31&5/12&13/31&12/31&10/31&20/61\\
32&16/35&16/37&16/39&13/33&17/43&14/33&18/47&7/18&22/51&5/11&5/11&11/26&14/33&13/33&1/3\\
34&17/37&17/39&17/41&2/5&2/5&3/7&3/7&1/3&2/5&5/11&16/35&26/59&3/7&3/7&2/5\\
36&6/13&18/41&18/43&2/5&15/37&3/7&16/37&14/37&2/5&8/19&11/24&17/37&3/7&16/37&28/67\\
38&19/41&19/43&19/45&19/47&16/39&20/51&17/39&23/55&1/3&9/22&27/61&6/13&6/13&13/30&17/39\\
40&20/43&4/9&20/47&20/49&17/41&21/53&7/16&18/41&16/41&24/61&17/41&6/13&19/41&31/69&7/16\\
42&7/15&21/47&3/7&7/17&5/12&18/43&8/19&19/43&25/61&1/3&5/12&29/67&13/28&20/43&32/73\\
44&22/47&22/49&22/51&22/53&2/5&19/45&23/59&4/9&4/9&2/5&26/67&19/45&32/71&7/15&7/15\\
46&23/49&23/51&23/53&23/55&19/47&20/47&24/61&4/9&21/47&19/47&1/3&11/26&20/47&7/15&22/47\\
48&8/17&24/53&24/55&8/19&20/49&3/7&3/7&27/65&22/49&10/23&28/71&5/13&3/7&34/77&15/32\\
50&25/53&5/11&25/57&25/59&7/17&3/7&22/51&7/18&9/20&23/51&7/17&1/3&3/7&22/51&37/81\\
52&26/55&26/57&26/59&26/61&22/53&2/5&23/53&9/23&31/71&24/53&32/75&30/77&11/28&13/30&23/53\\
54&9/19&27/59&27/61&3/7&5/12&23/55&24/55&30/71&30/73&5/11&5/11&23/55&1/3&35/83&24/55\\
56&28/59&28/61&4/9&28/65&28/67&8/19&7/16&25/57&2/5&5/11&26/57&8/19&22/57&2/5&7/16\\
58&29/61&29/63&29/65&29/67&29/69&25/59&31/73&26/59&23/59&34/79&27/59&37/83&25/59&1/3&37/89\\
60&10/21&6/13&30/67&10/23&30/71&26/61&2/5&27/61&33/79&9/22&11/24&28/61&26/61&24/61&36/91\\
62&31/65&31/67&31/69&31/71&31/73&3/7&3/7&4/9&4/9&33/83&38/85&29/63&39/89&3/7&1/3\\
64&32/67&32/69&32/71&32/73&32/75&3/7&28/65&4/9&29/65&2/5&37/87&6/13&6/13&28/65&2/5\\
66&11/23&33/71&33/73&11/25&3/7&28/67&29/67&35/83&30/67&12/29&5/12&6/13&31/67&29/67&41/97\\
68&34/71&34/73&34/75&34/77&34/79&29/69&10/23&2/5&31/69&39/89&36/91&41/93&32/69&44/97&10/23\\
70&35/73&7/15&5/11&35/79&35/81&30/71&31/71&37/87&9/20&32/71&37/93&30/71&13/28&33/71&31/71\\
72&12/25&36/77&36/79&4/9&36/83&31/73&7/16&32/73&40/91&33/73&30/73&11/26&5/11&34/73&46/103\\
74&37/77&37/79&37/81&37/83&37/85&32/75&38/89&11/25&13/31&34/75&42/97&13/33&44/101&7/15&7/15\\
76&38/79&38/81&38/83&38/85&38/87&3/7&3/7&34/77&2/5&5/11&5/11&40/101&3/7&7/15&36/77\\
\bottomrule
\end{tabular}
\end{table}
\end{landscape}

\begin{landscape}
\begin{table}
\caption{$\kappa(\{1,j,k\})$ when $j$ is even}
\label{kappa_j_even}
\centering \small
\begin{tabular}{*{16}{r}} \toprule
k/j&2&4&6&8&10&12&14&16&18&20&22&24&26&28&30\\ \midrule
3&1/4&2/7&1/3&4/11&5/13&2/5&7/17&8/19&3/7&10/23&11/25&4/9&13/29&14/31&5/11\\
4&1/3&2/5&2/5&1/3&4/11&5/13&2/5&2/5&7/19&8/21&9/23&2/5&2/5&11/29&12/31\\
5&1/3&1/3&2/7&1/3&1/3&6/17&7/19&8/21&9/23&2/5&11/27&12/29&13/31&14/33&3/7\\
6&2/7&2/5&3/7&3/7&3/8&1/3&2/5&7/17&8/19&3/7&3/7&2/5&11/27&12/29&13/31\\
7&1/3&3/8&4/13&1/3&4/11&3/8&1/3&6/17&7/19&3/8&11/29&12/31&13/33&2/5&15/37\\
8&1/3&1/3&3/7&4/9&4/9&2/5&4/11&1/3&5/13&3/7&10/23&11/25&4/9&4/9&8/19\\
9&3/10&2/5&2/5&3/10&6/19&5/13&2/5&2/5&1/3&8/21&9/23&2/5&2/5&11/29&12/31\\
10&1/3&4/11&3/8&4/9&5/11&5/11&5/12&5/13&4/11&1/3&3/8&7/17&4/9&13/29&14/31\\
11&1/3&2/5&5/12&7/19&1/3&4/13&2/5&7/17&5/12&12/31&1/3&2/5&11/27&12/29&5/12\\
12&4/13&5/13&1/3&2/5&5/11&6/13&6/13&3/7&2/5&8/21&5/13&1/3&10/27&2/5&3/7\\
13&1/3&5/14&3/7&3/7&5/14&8/25&1/3&11/29&8/19&3/7&3/7&14/37&1/3&16/41&13/31\\
14&1/3&2/5&2/5&4/11&5/12&6/13&7/15&7/15&7/16&7/17&9/23&2/5&2/5&1/3&12/31\\
15&5/16&3/8&3/7&7/16&2/5&3/8&5/16&10/31&7/19&3/7&10/23&7/16&17/41&3/8&1/3\\
16&1/3&2/5&7/17&1/3&5/13&3/7&7/15&8/17&8/17&4/9&8/19&2/5&11/27&7/17&9/23\\
17&1/3&7/18&9/23&4/9&4/9&5/13&7/18&1/3&6/19&8/21&16/39&11/25&4/9&4/9&19/47\\
18&6/19&7/19&8/19&5/13&4/11&2/5&7/16&8/17&9/19&9/19&9/20&3/7&9/22&12/29&13/31\\
19&1/3&2/5&2/5&4/9&9/20&13/31&2/5&2/5&12/37&1/3&9/23&2/5&4/9&13/29&9/20\\
20&1/3&8/21&3/7&3/7&1/3&8/21&7/17&4/9&9/19&10/21&10/21&5/11&10/23&5/12&2/5\\
21&7/22&2/5&9/22&12/29&5/11&5/11&2/5&9/22&5/13&7/22&14/43&2/5&11/27&3/7&14/31\\
22&1/3&9/23&3/7&10/23&3/8&5/13&9/23&8/19&9/20&10/21&11/23&11/23&11/24&11/25&11/26\\
23&1/3&3/8&5/12&12/31&5/11&11/24&16/37&7/17&5/12&3/8&1/3&8/25&19/49&12/29&5/12\\
24&8/25&2/5&2/5&11/25&7/17&1/3&2/5&2/5&3/7&5/11&11/23&12/25&12/25&6/13&4/9\\
25&1/3&5/13&11/26&13/33&3/7&6/13&6/13&17/41&8/19&11/26&5/13&16/49&1/3&11/29&13/31\\
26&1/3&2/5&11/27&4/9&4/9&10/27&2/5&11/27&9/22&10/23&11/24&12/25&13/27&13/27&13/28\\
27&9/28&11/28&3/7&3/7&15/37&6/13&13/28&19/43&2/5&3/7&3/7&11/28&9/28&18/55&12/31\\
28&1/3&11/29&12/29&4/9&13/29&2/5&1/3&7/17&12/29&5/12&11/25&6/13&13/27&14/29&14/29\\
29&1/3&2/5&3/7&13/30&5/13&18/41&7/15&7/15&20/47&3/7&13/30&22/53&2/5&1/3&10/31\\
30&10/31&12/31&13/31&8/19&14/31&3/7&12/31&9/23&13/31&2/5&11/26&4/9&13/28&14/29&15/31\\
31&1/3&2/5&13/32&7/16&16/41&18/43&7/15&15/32&22/49&7/17&10/23&7/16&13/32&23/59&20/61\\
32&1/3&13/33&14/33&2/5&5/11&5/11&13/33&1/3&8/19&14/33&11/27&3/7&13/29&7/15&15/31\\
33&11/34&13/34&7/17&15/34&18/43&2/5&21/47&8/17&8/17&23/53&2/5&11/25&15/34&7/17&13/34\\
34&1/3&2/5&3/7&3/7&5/11&16/35&5/12&2/5&5/13&3/7&3/7&12/29&13/30&14/31&15/32\\
35&1/3&7/18&5/12&4/9&4/9&5/13&3/7&8/17&17/36&5/11&8/19&25/59&4/9&4/9&5/12\\
36&12/37&2/5&3/7&16/37&10/23&17/37&11/25&15/37&1/3&3/7&16/37&2/5&13/31&7/16&5/11\\
37&1/3&15/38&8/19&4/9&17/38&19/49&7/17&24/53&9/19&9/19&26/59&25/61&4/9&17/38&29/67\\
38&1/3&5/13&16/39&17/39&5/12&6/13&6/13&16/39&11/28&8/21&10/23&17/39&11/27&14/33&15/34\\
39&13/40&2/5&17/40&20/47&9/20&7/17&2/5&24/55&9/19&19/40&28/61&3/7&2/5&13/29&9/20\\
40&1/3&16/41&17/41&18/41&2/5&6/13&19/41&3/7&17/41&1/3&13/31&18/41&14/33&12/29&3/7\\
\bottomrule
\end{tabular}
\end{table}
\end{landscape}


\begin{thebibliography}{99}

\itemsep 0pt


\bibitem{7runners}
J. Barajas and O. Serra,
{\em The lonely runner with seven runners},
The Electronic Journal of Combinatorics, 15 (2008), \#R48.


\bibitem{lonely} W. Bienia, L. Goddyn, P. Gvozdjak, A. Seb\H{o} and
M. Tarsi, {\em Flows, view obstructions, and the lonely runner},
J. Combin. Theory (B), 72 (1998), 1 -- 9.

\bibitem
{6runners1}
T. Bohman, R. Holzman, D. Kleitman,
{\em Six lonely runners}, Electronic J. Combinatorics, 8 (2001), Research Paper 3, 49 pp.

\bibitem{CG} D. G. Cantor and B. Gordon,
{\em Sequences of integers with missing differences},
 J. Combin. Th. (A), 14 (1973),  281 -- 287.

\bibitem{conjecture}
J.~Carraher, D.~Galvin, S.~Hartke, A.J.~Radcliffe, and D.~Stolee.
{\em On the independence ratio of distance graphs}, 
Discrete Math., 339 (2016), 3058 -- 3072. 


\bibitem{chz} G. J. Chang, L. Huang and X. Zhu,
        {\em The circular chromatic numbers and the fractional
        chromatic numbers of distance graphs},
        Europ. J. of Combin., 19 (1998), 423 -- 431.

\bibitem{clz} G. Chang, D. Liu and X. Zhu,
{\em Distance graphs and $T$-coloring}, J. Combin. Th. (B), 75 (1999) 159 -- 169.

\bibitem{chen2} Y. G. Chen, {\em On a conjecture about Diophantine
approximations, II}, J. Number Theory, 37 (1991), 181--198.

\bibitem{chen3} Y. G. Chen, {\em On a conjecture about Diophantine
approximations, III}, J. Number Theory, 39 (1991), 91--103.

\bibitem{chen4} Y. G. Chen, {\em On a conjecture about Diophantine
approximations, IV}, J. Number Theory, 43 (1993), 186--197.

\bibitem{cch} J. Chen, G. J. Chang and K. Huang,
 {\em Integral distance graphs},
 J. Graph Theory, 25 (1997), 287 -- 294.

\bibitem{Collister}
D. Collister and D.~D.-F.~Liu, 
{\em Study of $\kappa(D)$ for $D = \{2,3 , x, y\}$}, 
J. Kratochv�l, M. Miller, and D. Froncek (Editors). Combinatorial Algorithms. Lecture Notes in Computer Science, Springer. Proceeding of the 25th International Workshop, 2014, Duluth, MN, 250 -- 261.

\bibitem{C3}
T. W. Cusick,
{\em View-obstruction problems in $n$-dimensional geometry},
J. Combin. Theory Ser. A, 16 (1974), 1 -- 11.

\bibitem{cp} T. W. Cusick and C. Pomerance,
{\em View-obstruction problems, III}, J. Number Th., 19 (1984), 131--139.

\bibitem{realline} R. B. Eggleton, P. Erd\H{o}s and D. K. Skilton,
 {\em Colouring the real line},
 J. Combin.~Theory (B), 39 (1985), 86 -- 100.

\bibitem{EES86}
R.~B.~Eggleton, P.~Erd\H{o}s and D.~K.~Skilton, {\em Research
problem 77}, Disc.~Math., 58 (1986), 323.


\bibitem{ees90} R. B. Eggleton, P. Erd\H{o}s and D. K. Skilton,
 {\em Colouring prime distance graphs},
 Graphs and Combinatorics, 6 (1990), 17 -- 32.

\bibitem{GL} J. Griggs and D.~D.-F.~Liu, {\em The channel assignment problem for mutually adjacent sites},  J. Combin. Th. (A), 68 (1994), 169--183.

\bibitem{Gupta}
S. Gupta,
{\em Sets of Integers with Missing Differences},
J. Combin. Th. (A), 89 (2000), 55--69.


\bibitem{H} N. M. Haralambis, {\em Sets of integers with missing differences},  J. Combin. Th. (A), 23 (1977), 22 -- 33.

\bibitem{KK} A. Kemnitz and H. Kolberg, {\em Coloring of integer
distance graphs}, Disc.~Math., 191 (1998), 113 -- 123.

\bibitem{KM} A. Kemnitz and M. Marangio, {\em Colorings and list colorings of integer distance graphs}, Congr. Numer., 151 (2001), 75 -- 84.

\bibitem{KM2}  A. Kemnitz and M. Marangio, {\em Chromatic numbers of integer distance graphs}, Disc.~Math., 233 (2001), 239--246.

\bibitem{LL} 
P. Lam and W. Lin, {\em Coloring distance graphs with intervals
as distance sets}, European J. of Combin., 26 (2005),  1216 -- 1229.

\bibitem{llz}
K.-W.~Lih, D.~Liu and X.~Zhu. 
{\em Star-extremal circulant graphs}. 
SIAM~J. Disc.~Math., 12 (1999), 491 -- 499. 

\bibitem{LL2}
W. Lin, P. Lam and Z. Song, {\em Circular chromatic numbers of some distance graphs}, Disc.~Math., 292 (2005), 119 -- 130.

\bibitem{L} D. Liu, {\em $T$-coloring and chromatic number of distance graphs}, Ars~Combin., 56 (2000), 65--80.
\bibitem{surveyL} D.~Liu, {\em From rainbow to the lonely runner: A survey on coloring parameters of distance graphs},  Taiwanese J. Math., 12 (2008), 851 -- 871.

\bibitem
{Aileen}
D. Liu and A. Sutedja,
{\em Chromatic number of distance graphs generated by the sets $\{2, 3, x, y\}$},
J. of Combin. Optimization, 25 (2013), 680 -- 693.

\bibitem{lz97} D. Liu and X. Zhu,
 {\em Distance graphs with missing multiples in the distance sets},
 J. Graph Theory, 30 (1999), 245 -- 159.

\bibitem{LZ} D. Liu and X. Zhu, {\em Fractional chromatic number and circular chromatic number for distance graphs with large clique size}, J. Graph Theory, 47 (2004), 129 -- 146.


\bibitem
{PT}
R.~K.~Pandey, A. Tripathi.
{\em On the density of integral sets with missing differences from sets related to arithmetic progressions}. 
J. Number Theory, 131 (2011), 634 -- 647.

\bibitem{RP} J. Rabinowitz and V. Proulx,
{\em An asymptotic approach to the channel assignment problem}, SIAM J. Alg. Disc. Methods, 6 (1985), 507--518.


\bibitem{voigt} M. Voigt, {\em Colouring of distance graphs},
Ars Combinatoria, 52 (1999), 3 -- 12.

\bibitem{vw} M. Voigt and H. Walther,
 {\em Chromatic number of prime distance graphs},
 Discrete Appl. Math., 51 (1994), 197 -- 209.

\bibitem{vw2} M. Voigt and H. Walther,
{\em On the chromatic number of special distance graphs},
Disc.~Math., 97 (1991), 395 -- 397.


\bibitem{W} J. M. Wills, {\em Zwei S\"{a}tze \"{u}ber inhomogene
diophantische appromixation von irrationlzahlen},
Monatsch. Math., 71 (1967), 263 -- 269.

\bibitem{WL} 
J. Wu and W. Lin, {\em Circular chromatic numbers and fractional
chromatic numbers of distance graphs with distance sets missing an interval}, Ars Comb., 70 (2004), 161 -- 168.

\bibitem{pattern}
X. Zhu,
{\em Pattern periodic coloring of distance graphs},
J. Combin. Theory (B), 73 (1997), 195 -- 206.



\bibitem{d=3} X. Zhu,
{\em The circular chromatic number of distance graphs with distance sets of cardinality $3$}, J. Graph Theory,  41 (2002)  195 -- 207.

\bibitem{zhu03} X. Zhu,
{\em The circular chromatic number of a class of distance graphs}, Disc.~Math., 265/1-3 (2003), 337 -- 350.

\end{thebibliography}
\end{document}